\documentclass[reqno,a4paper]{amsart}  
\usepackage{amsmath,amssymb,amsfonts}     
\usepackage[linkbordercolor={1 1 1}]{hyperref}
\usepackage{color}
\usepackage{enumitem}
\usepackage{tikz}
\usepackage{mathrsfs}

\theoremstyle{plain}      
 
\newtheorem{thm}{Theorem}     
\newtheorem{theorem}[thm]{Theorem}     
     
\newtheorem{corollary}[thm]{Corollary}     
     
\newtheorem{lemma}[thm]{Lemma}     
     
\newtheorem{proposition}[thm]{Proposition}     
     
\theoremstyle{remark}      
\newtheorem{example}[thm]{Example} 
\newtheorem{remark}[thm]{Remark}

\theoremstyle{definition}      
     
\newtheorem{definition}[thm]{Definition}     

\def\al{{\alpha}}

\def\de{{\delta}}
\def\De{{\Delta}}
\def\om{{\omega}}

\def\si{{\sigma}}

\def\epsilon{{\varepsilon}}

\DeclareMathAlphabet{\doba}{U}{msb}{m}{n}         
\gdef\mC{\doba{C}}

\gdef\mN{\doba{N}}

\gdef\mQ{\doba{Q}}
\gdef\mR{\doba{R}}

\gdef\mZ{\doba{Z}}

\def\rH{\mathscr{H}}
\def\rV{\mathscr{V}}

\def\vol{{\mathop{\rm vol}}}
\def\Scal{{\mathop{\rm Scal}}}
\def\SO{{\mathop{\rm SO}}}

\def\card{{\mathop{\rm card}}}
\def\id{{\mathop{\rm id}}}
\def\di{{\,\rm d}}

\let\scal\Scal
\def\grad{{\mathop{\rm grad}}}

\let\pa\partial     
\let\na\nabla     
\let\<\langle 
\let\>\rangle 
\let\ti\tilde
\let\witi\widetilde

\newcommand{\definedas}{\mathrel{\raise.095ex\hbox{\rm :}\mkern-5.2mu=}}

\begin{document}     



\title{The $\boldsymbol{S^1}$-equivariant Yamabe invariant of 3-manifolds}
\def\emailaddrname{{\itshape E-mail}}
\author{Bernd Ammann} 
\address{Fakult\"at f\"ur  Mathematik \\ 
Universit\"at Regensburg \\
93040 Regensburg \\  
Germany}
\email{bernd.ammann@mathematik.uni-regensburg.de}

\author{Farid Madani}
\address{Institut f\"ur  Mathematik \\ 
Goethe Universit\"at Frankfurt \\
60325 Frankfurt am Main\\  
Germany}
\email{madani@math.uni-frankfurt.de}

\author{Mihaela Pilca}
\address{Fakult\"at f\"ur  Mathematik \\ 
Universit\"at Regensburg \\
93040 Regensburg \\  
Germany}

\email{mihaela.pilca@mathematik.uni-regensburg.de}

\begin{abstract}
We show that the $S^1$-equivariant Yamabe invariant of the $3$-sphere, endowed with the Hopf action, is equal to the (non-equivariant) Yamabe invariant of the $3$-sphere. More generally, we establish a topological upper bound for the $S^1$-equivariant 
Yamabe invariant of any closed oriented $3$-manifold endowed with an $S^1$-action.  
Furthermore, we prove a convergence result  for the equivariant Yamabe constants of an accumulating sequence of subgroups of a compact Lie group acting on a closed manifold.
\end{abstract}

\subjclass[2010]{}
%

\date{\today}


\maketitle

\setcounter{tocdepth}{1}

\section{Overview over the classical Yamabe invariant}

The Yamabe constant $\mu(M, [g])$ of an $n$-dimensional conformal compact 
manifold $(M, [g])$ is the infimum of the restriction to the conformal class $[g]$ of the Einstein--Hilbert functional defined on the set of all Riemannian metrics as
$$h\longmapsto \frac{\int_M \scal_h\,dv_h}{\vol(M,h)^{\frac{n-2}{n}}}.$$ 
Aubin \cite{aubin:76} 
proved that the Yamabe constant of $(M,[g])$ is bounded above by the 
Yamabe constant of the sphere,  {\it i.e.}\/ 
$\mu(M,[g])\leq \mu(S^n,[g_{st}])$.
The Yamabe invariant $\sigma(M)$ of a compact manifold $M$ is defined as 
 \[\si(M):=\sup_{[g]\in C(M)}\mu(M,[g]),\]
where $C(M)$ is the set of all conformal classes on $M$. 
It follows that $\sigma(M)\leq \sigma(S^n)=\mu(S^n,[g_{st}])$. 
In particular, the Yamabe invariant of any compact manifold is finite. 
The Yamabe invariant $\si(M)$ is positive if and only if a metric of positive 
scalar curvature exists on $M$.

In dimension~$2$, the Yamabe invariant is a multiple of the Euler 
characteristic. For $n\geq 3$ it is in general 
a difficult problem to compute the Yamabe invariant, and 
only in few cases it can be calculated explicitly.
Aubin \cite{aubin:76} proved for the $n$-dimensional 
sphere $\sigma(S^n)=\mu(S^n,[g_{st}])=n(n-1)(\vol(S^n,g_{st}))^{2/n}$.
Kobayashi \cite{kobayashi:87} and  Schoen \cite{schoen:89} proved that 
$\sigma(S^{n-1}\times S^1)=\sigma(S^n)$. For many closed 
manifolds $M$, one can show 
$\si(M)=0$ as the existence of metrics with 
positive scalar curvature is obstructed, whereas conformal classes
$[g_i]$ with $\mu(M,[g_i])\to 0$ can be written down explicitly. 
For example the~$n$-torus $T^n$ does not carry a metric of positive scalar 
curvature which can be shown with enlargeability type index obstructions 
by Gromov and Lawson or with the hypersurface obstruction by Schoen and Yau.
For the standard metric $g_0$ we have $\mu(T^n,[g_0])=0$, so $\si(T^n)=0$.
Similarly we know $\si(M)=0$ for all nilmanifolds, and quotients thereof.

In order to determine non-zero values for $\si$, many modern techniques 
were used: Ricci-flow, Atiyah-Singer index theorem, Seiberg-Witten theory, 
and the Bray-Huisken inverse mean curvature proof of the Penrose inequality. 
In dimension~$3$, values for the Yamabe 
invariant of irreducible manifolds were already conjectured and 
partially studied in \cite{anderson.msri, And}. 

For example, on a hyperbolic 
$3$-manifold~$H^3/\Gamma$ 
the supremum in the definition of the Yamabe invariant $\si(M)$ 
is attained in
the conformal class of the hyperbolic metric $g_{\rm hyp}$, 
and the infimum in the definition of 
$\mu(H^3/\Gamma,[g_{\rm hyp}]$ is attained in $g_{\rm hyp}$.
More generally, it follows from Perelman's work on the Ricci flow that 
for $3$-manifolds with $\si(M)\leq 0$, the value of $\si(M)$ is determined
by the volume of the hyperblic pieces in the Thurston decomposition. 
We learned this from~\cite[Prop.~93.10 on page 2832]{kleiner.lott:08}, 
but ideas for this application go back to \cite{anderson:06}. 
In the case $\si(M)>0$, 
$n=3$, $M$ is the connected sum of copies of quotients 
$S^2\times S^1$ and of quotients
of $S^3$. For connected sums of copies of $S^2\times S^1$ we have 
$\si(M)=\si(S^3)$ but the precise value cannot be determined in most cases.
Using inverse mean curvature flow, the Yamabe invariants of $\mR P^3$ and 
some related spaces were determined in \cite{bray.neves:04} 
and \cite{akutagawa.neves:07}, e.g.\ $\sigma(\mathbb RP^3)= 2^{-2/3}\sigma(S^3)$.
This is indeed a special case of Schoen's conjecture explained below.

Also in higher dimensions the case of positive Yamabe invariant is 
notoriously difficult. 
In dimension $n\geq 5$ one does not know any $n$-dimensional 
manifold $M$ for which one can prove $0<\si(M)<\si(S^n)$. 
In dimensions $n\leq 4$ there are some examples for which exact 
calculations can be carried out, even in the positive case. 
The values for $\mathbb CP^2$ and some related spaces
were calculated by LeBrun \cite{lebrun:97} using Seiberg-Witten theory. 
The calculation then was simplified considerably by Gursky and 
LeBrun \cite{gursky.lebrun:98}. This proof no longer uses 
Seiberg-Witten theory, but only the index theorem by Atiyah and Singer. 
See also \cite{gursky.lebrun:98,lebrun:96,lebrun:99a} for related results.

Recently, surgery techniques known from the work of Gromov and Lawson could be refined to obtain explicit positive lower bounds
for the Yamabe invariant. 
Such bounds are easily obtained for \emph{special} manifolds, 
e.\ g.\ for manifolds
with Einstein metrics or connected sum of such manifolds. Namely, 
a theorem by Obata \cite{obata:71.72} states that 
the Einstein--Hilbert functional of an Einstein metric $g$ 
equals $\mu(M,[g])$, thus providing a lower bound for $\sigma(M)$. 
For instance,  if $M$ is $S^n$, $T^n$, $\mathbb RP^n$ or $\mathbb CP^n$, 
the canonical Einstein metrics provide lower bounds
for $\sigma(M)$.  However, obtaining 
a lower bound for $\si(M)$ is difficult in general
if $M$ carries a metric of positive scalar curvature but no Einstein metric. Using surgery theory, Petean and Yun have proven that $\si(M)\geq 0$ for all simply-connected manifolds of dimension at least $5$, see 
\cite{petean:03}, \cite{petean.yun:99}. 
Stronger results can be obtained with the surgery formula
developed in \cite{ammann.dahl.humbert:13b}. 
For example, it now can be shown, see \cite{ammann.dahl.humbert:13} and 
\cite{ammann.dahl.humbert:13c,ammann.dahl.humbert:15}, 
that simply-connected manifolds of dimension $5$ resp.\
$6$ satisfy $\si(M)\geq 45.1$ resp.\ $\si(M)\geq 49.9$.

In order to find more manifolds with $0<\si(M)<\si(S^n)$, it would be helpful
to prove the following conjecture by Schoen \cite{schoen:89}: it states 
that if $\Gamma$ is a finite group acting freely on $S^n$, then 
$\si(S^n/\Gamma)=\sigma(S^n)/(\# \Gamma)^{2/n}$. In particular, it would 
imply with \cite{ammann.dahl.humbert:13b} that for any odd $n\geq 5$ and
sufficiently large $k:=\# \Gamma$, every manifold $M$ representing the 
bordism class $[S^n/\Gamma]\in \Omega^{\rm spin}_n(B\Gamma)$ with maps 
inducing isomorphisms $\pi_1(M)\cong \Gamma\cong \pi_1(S^n/\Gamma)$ has 
$\si(M)=\sigma(S^n)/(\# \Gamma)^{2/n}$ an many more similar conlcusions.
Unfortunately, besides the trivial cases $\Gamma=\{\id\}$ or $n=2$, 
this conjecture has only been
proven in the particular case, when $n=3$ and $\#\Gamma=2$,
which is the determination of $\si(\mR P^3)$ 
by Bray and Neves in \cite{bray.neves:04} mentioned above.

\section{Overview over the $G$-equivariant Yamabe invariant}

In this paper, we study the $G$-equivariant setting by taking the supremum 
and the 
infimum only among $G$-invariant metrics and conformal classes 
where $G$ is a compact Lie group acting on $M$, see Section~\ref{def.G.yamabe}
for details. 
The associated invariants are called the  $G$-equivariant Yamabe constant or 
simply the $G$-Yamabe constant $\mu(M,[g]^G)$, 
and similarly the  $G$-(equivariant) Yamabe invariant $\si^G(M)$. 
To our knowledge the first reference for the $G$-equivariant Yamabe constant
$\mu(M,[g]^G)$
is  B\'erard Bergery \cite{berardbergery}. In particular, he formulated a
$G$-equivariant version of the Yamabe conjecture, 
which was the main subject of an article by Hebey and 
Vaugon \cite{hebey.vaugon:93} and by the second
author \cite{madani:10, madani:12}.
In general neither $\si^G(M)\leq \si(M)$ nor $\si^G(M)\geq \si(M)$, see Example~\ref{ex.diff}.

One motivation for the present article is to shed new light on 
Schoen's conjecture which is equivalent to saying 
$\sigma^\Gamma(S^n)=\sigma(S^n)$. 
A proof of Schoen's conjecture (or even partial results) 
would be very helpful, as it would provide 
interesting conclusions about the Yamabe invariant of non-simply connected 
manifolds. 
For example, if we were able to obtain an upper bound on 
$\si^\Gamma(S^n)$ which is uniform in $\Gamma$, then 
the Yamabe invariant would define interesting subgroups of
the spin bordism and oriented 
bordism groups, see \cite{ammann.dahl.humbert:13b}. 

The simplest case
of Schoen's conjecture is when $\mZ_k\subset S^1\subset \mC$ acts by complex 
multiplication on $S^3\subset \mC^2$, the so-called Hopf action.
As it seems currently 
out of reach to show $\sigma^{\mZ_k}(S^n)=\sigma(S^n)$ for $k>2$, 
we study the limit $k\to \infty$ instead, and this leads two the following two 
questions:
\begin{enumerate} 
\item Is $\si^{S^1}(S^3)=\si(S^3)$ true for the Hopf action? 
\item Assume that a sequence $(H_i)$ of subgroups of $G$ ``converges'' 
to $G$. Can we conclude that 
$\si^{H_i}(M)$ converges to $\si^G(M)$?  
\end{enumerate}

The answer to the first question is answered affirmatively by our main theorem.
More generally, we give an upper bound for the $S^1$-Yamabe invariant of any $3$-dimensional closed oriented manifold $M$, endowed with an $S^1$-action. This upper bound depends only on the following topological invariants: the first Chern class of the associated line bundle and the Euler--Poincar\'e characteristic of the quotient space (see Theorem \ref{mainthm} for the precise statement). 

Our strategy is to use the quotient space $M/S^1$. We  distinguish the following three cases, since the isotropy group of any point is either $\{\id\}$, $\mZ_k$ or $S^1$.  
If the $S^1$-action has at least one fixed point, a result of Hebey and Vaugon \cite{hebey.vaugon:93} implies that $\sigma^{S^1}(M)\leq \sigma(S^3)$. 
If the $S^1$-action is free, then $M/S^1$ is a smooth surface.  
In order to find an upper bound in this case, we mainly use O'Neill's 
formula relating the curvatures  of the total space 
and the base space of a Riemannian submersion and the Gau\ss--Bonnet theorem. 
In the last case, when the $S^1$-action is neither free nor has fixed points 
({\it i.e.} there exists at least one point with non-trivial finite isotropy 
group), the quotient space $M/S^1$ is a closed $2$-dimensional orbifold. 
We proceed as in the free action case, since the Gau\ss--Bonnet theorem 
still holds on orbifolds (see \cite{satake:57}). 
In the two latter cases, we find a topological upper bound of $\sigma^{S^1}(M)$,
which depends only on the Euler--Poincar\'e characteristic of $M/S^1$ and 
the first Chern number of the associated line bundle over $M/S^1$. 

The last part of the article partially answers the second question. 
More precisely the statement of Corollary~\ref{approx.cor} is
  $$\liminf_{i\to \infty}\si^{H_i}(M)\geq  \si^{G}(M).$$
Unfortunately, the corresponding  $\leq$-inequality which would allow the 
interesting application to Schoen's conjecture still fails due to lack of 
curvature control.

\section{Preliminaries, definitions and some known results}

\subsection{Definition of the $G$-equivariant Yamabe invariant}\label{def.G.yamabe}

In this section we assume that a compact Lie group $G$ acts on the compact manifold $M$. All actions are supposed
to be smooth.

We recall that the Einstein-Hilbert functional of $M$ is given by 
\begin{equation}\label{def.J}
J(\ti g):=\frac{\int_M \scal_{\ti g} \di v_{\ti g}}{\vol(M,\ti g)^{\frac{n-2}{n}}}.
\end{equation}

We denote by $[\ti g]^{G}$ the set of $G$-invariant metrics in the conformal class of $\ti g$ and by $C^{G}(M)$ the set of all conformal classes containing at least one $G$-invariant metric.
\begin{definition}[$G$-Yamabe invariant]\label{defyamabeinv}
We define the $G$-equivariant Yamabe constant (or shorter: the $G$-Yamabe constant) by
\begin{equation}
\mu(M,[\ti g]^{G})=\inf_{g'\in [\ti g]^{G}} J(g')
\end{equation}
and the $G$-equivariant Yamabe invariant of $M$ 
(or shorter: the $G$-Yamabe invariant) by
$$\sigma^{G}(M)=\sup_{[\ti g]^{G}\in C^{G}(M)} \mu(M,[\ti g]^{G})\in (-\infty,\infty].$$
\end{definition}

\begin{remark}It follows for the solution of the equivariant Yamabe problem \cite{hebey.vaugon:93} that $\mu(M,[g]^G)>0$ if and only if $[g]$ contains 
a $G$-invariant metric of positive scalar curvature. It thus follows that 
$\si^G(M)>0$ holds if and only if $M$ carries a $G$-invariant metric 
of positive scalar curvature.
\end{remark}

The following examples show that both $\si^G(M)>\si(M)$ and $\si^G(M)<\si(M)$ may arise.

\begin{example}\label{ex.diff}
 $\si^G(M)\leq \si(M)$ nor $\si^G(M)\geq \si(M)$. For example if $S^1$ acts on the $S^1$ factor
of $N\times S^1$, $\dim N=n-1$, and if $N$ is a compact manifold carrying a metric of positive scalar curvature,
then $\si^{S^1}(N\times S^1)=\infty$, whereas $\si(N\times S^1)\leq \si(S^n)< \infty$. On the other hand, if $M$ is a 
simply-connected circle bundle over a K3-surface then $\si(M)>0$ but  $\si^{S^1}(M)=0$. Here $\si(M)>0$ follows 
classically from work by Gromov and Lawson and the fact that every compact simply connected spin $5$-manifolds
is a spin boundary. For $\si^{S^1}(M)\leq 0$ we refer to  
\cite[Theorem~6.2]{wiemeler:1506.04073}. The inequality $\si^{S^1}(M)\geq 0$ follows 
from~\eqref{oneillscal} by taking an $S^1$-invariant metric $g_1$ on $M$, we rescale 
the fibers by a factor $\ell>0$ and obtain $g_\ell$ and then  $\lim_{\ell\to 0}\mu(M,[g_\ell]^{S^1})=0$.
\end{example}

The situation changes in the non-positive case. In the case $\si^G(M)\leq 0$ we have $\mu(M,[g]^G)=
\mu(M,[g])$ for any $G$-invariant conformal class $[g]$, as the maximum principle implies that minimizers are 
unique up to a constant. Thus $\si(M)\geq \si^G(M)$ in this case.

\subsection{Some known results}

In \cite{hebey.vaugon:93}, Hebey and Vaugon gave the following upper bound for the $G$-Yamabe constant:
\begin{proposition}[Hebey--Vaugon]\label{Prop-HV}
Let $M$ be an $n$-dimensional compact connected 
oriented manifold endowed with an action of a compact Lie group $G$,
admitting at least one orbit of finite cardinality. 
Then the following inequality holds:
\[\sigma^{G}(M)\leq \sigma(S^n)\bigl(\inf_{p\in M}\card(G\cdot p)\bigr)^{\frac 2 n}.\]
\end{proposition}

Other results in the literarture can be rephrased as follows.

\begin{proposition}[B\'erard Bergery, \cite{berardbergery}]
If $G$ is a compact Lie group whose connected component of the identity is non-ablian and which acts effectively on a closed manifold $M$ with cohomogeneity $2$. Then $\si^G(M)>0$.
\end{proposition}

\begin{proposition}[B\'erard Bergery \cite{berardbergery} $n=3$, Wiemeler \cite{wiemeler:1305.2288} all $n$]
Let an abelian Lie group $G$ act effectively on a closed connected manifold $M$ with a fix point component of codimension $2$. Then $\si^G(M)>0$.
\end{proposition}

More recent progress about the question whether  $\si^G(M)>0$ can be found in \cite{hanke:08} and 
\cite{wiemeler:1506.04073}.

\subsection{Scalar curvature of $S^1$-bundles}\label{sec.s1.curv}

Let $M^n$ be a compact oriented and connected manifold, 
which is an $S^1$-bundle over $N$, 
let $\pi: M\to N$ be the projection, let~$\tilde g$ be an $S^1$-invariant 
metric on~$M$ and~$g$ its projection under $\pi$ on $N$. Let $K$ denote the tangent vector field induced by the $S^1$-action and let $\ell$ be its length 
(with respect to $\ti g$) and $e_0:=\frac{K}{\ell}$. We define the $(2,1)$-tensor fields $A$ and $T$ on $M$ as in \cite[9.C.]{besse:87}, {\it i.e.} for all vector fields $U,V$ on $M$: 
\[A_U V=\rH \nabla_{\rH U}\rV V +\rV \nabla_{\rH U}\rH V,\]
\[T_U V=\rH \nabla_{\rV U}\rV V +\rV \nabla_{\rV U}\rH V,\]
where $\rH$ and $\rV$ denote the horizontal, resp. vertical part of a vector field. The tensor $A$ measures the non-integrability of the horizontal distribution, whereas $T$ is essentially the second fundamental form of the $S^1$-orbits.
{\it Cf.} \cite[9.37]{besse:87}, the following formula relating the scalar curvatures of $(M,\ti g)$ and $(N,g)$ holds:
$$\witi\scal = \scal- |A|^2 - |T|^2-|T_{e_0}e_0|^2-2\check{\de} (T_{e_0}e_0),  $$
where $\check\de$ is the codifferential in the horizontal direction.
 For any vector fields $X, Y$ on $N$ with horizontal lifts $\witi X, \witi Y$, the vertical part of $[\witi X,\witi Y]$ equals $\Omega(X,Y)K:=2A_{\witi X}{\witi Y}$.
We compute:
$$|A|^2= \frac{\ell^2}4 |\Omega|^2,\quad T_K\witi X= \na_K {\witi X}=\na_{\witi X} K= \frac{\pa_{\witi X}\ell}\ell K, \quad T_{e_0}e_0= -\frac{\grad\ell}\ell=-\grad\log \ell,$$
which yield
  \begin{equation}\label{oneillscal}
 \scal_{\ti g}
      = \scal_g - \frac{\ell^2}4 |\Omega|^2 - 2 \frac{|d\ell|^2}{\ell^2}
   + 2\De_g(\log \ell)
= \scal_g- \frac{\ell^2}4 |\Omega|^2 + 2\frac{\De_g \ell}{\ell},
  \end{equation}
where $\Delta_g$ is the Laplacian of the base $(N,g)$.

\subsection{An analytical ingredient}

We recall that the following classical result still holds on orbifolds:
\begin{lemma}\label{lem laplace}
Let $(\Sigma, g)$ be a closed $2$-dimensional orbifold. Let $f\in C^k(\Sigma)$ be a function with $\int_\Sigma f dv_g=0$. Then there exists a solution $u\in C^{k+2}(\Sigma)$  of the equation $\Delta_g u=f$, which is unique up to an additive constant.
\end{lemma}
The proof of Lemma~\ref{lem laplace} is analogously to the classical case.

\section{The $S^1$-Yamabe invariant}

In this section we always have $G=S^1$, and we use the notation $N=M/S^1$ similar to Section~\ref{sec.s1.curv}.
Here $N$ may have singular points, i.e. orbifold points or boundary points.

\subsection{Yamabe functional on $S^1$-bundles}

If the action of $S^1$ is free, then 
by \eqref{oneillscal}, we obtain from \eqref{def.J}:
\begin{equation}\label{Yam func}
  J(\tilde g)=\frac{2\pi\int_{N}(\scal_g-\frac{\ell^2}{4}|\Omega|_g^2)\ell \di v_g}{(\int_{N} 2\pi \ell \di v_g)^{\frac{n-2}{n}}},
\end{equation}
since the length of any fibre is $2\pi\ell$.   
The Yamabe functional of $(M,\ti g)$ is the restriction of the Einstein-Hilbert functional to the conformal class of $\ti g$.
It can be equivalently written as follows:
\begin{equation}\label{Yamfunctional}
  J(u^{\frac{4}{n-2}}\tilde g)=\frac{ \int_{N} 2\pi \ell\bigl(\frac{4(n-1)}{n-2} |du|_g^2 
  + \Scal_{\ti g} u^2\bigr) \di v_g }
{\bigl(\int_N2\pi\ell u^{\frac{2n}{n-2}}\di v_g\bigr)^{\frac{n-2}{n}}},
\end{equation}
where $\scal_{\ti g}$ is given by \eqref{oneillscal}.

\subsection{Classification of $3$-manifolds with $\si^{S^1}(M)>0$}\label{pos.mu.s1}
It is completely understood, under which condition there is an $S^1$-invariant metric of positive sclar curvature,
in other words, when  $\si^{S^1}(M)>0$.

\begin{theorem}[{\cite[Theorem 12.1]{berardbergery}}]
Let $M$ be a compact connected $3$-dimensional manifold with a 
smooth $S^1$-action on~$M$.
\begin{enumerate}[label=\alph*)]
\item If the action has a fixed point, then $\si^{S^1}(M)>0$.
\item If the action has no fixed point, then $\si^{S^1}(M)>0$ if and only if
$M$ is a finite quotient of $S^3$ or of $S^2\times S^1$.
\end{enumerate}
\end{theorem}
Note that every finite quotient of $S^3$ by a freely acting 
subgroup of $\SO(4)$ admits a 
non-trivial $S^1$-action \cite[Sec.~6, Theorem~5]{orlik:72}.

\subsection{Oriented $3$-manifolds}\label{subsec dim3}
From now on, we assume that~$M$ is a $3$-dimensional compact oriented connected manifold endowed with an $S^1$-action.
If this $S^1$-action has at least one fixed point, 
Proposition~\ref{Prop-HV} implies that the Yamabe invariant of $S^3$ 
is an upper bound for the $S^1$-Yamabe invariant: 
$\sigma^{S^1}(M)\leq \sigma(S^3)$.

We want to determine an upper bound for the $S^1$-Yamabe invariant 
in the complementary case, i.e.\ we consider $S^1$-actions without 
fixed points. This implies that~$M$ is an $S^1$-principal (orbi)bundle 
over $\Sigma:=M/S^1$, which is a $2$-dimensional orbifold  
(a smooth surface, if the action is free). 
As usually, we use the correspondence between $S^1$-principal bundles and complex line bundles defined by $$\Sigma\mapsto L:=\Sigma\times_{S^1}\mC.$$

We write $c_1(L,\Sigma):=\langle c_1(L), [\Sigma] \rangle\in\mQ$, 
where $c_1(L)\in H^2(\Sigma,\mQ)$ is the first rational Chern class of $L$ in the orbifold sense. Let $\chi(\Sigma)=c_1(T\Sigma,\Sigma)$ be the (orbifold) Euler-Poincar\'e characteristic of $\Sigma$.

We are now ready to state our main result:

\begin{theorem}\label{mainthm}
Let $M$ be a $3$-dimensional compact connected 
oriented manifold endowed with an $S^1$-action without fixed points. With the above notation, the following assertions hold:
\begin{enumerate}[label=\roman*)]
\item \label{item i}If $\chi(\Sigma)>0$ and $c_1(L,\Sigma)\neq 0$, then  
$$0<\sigma^{S^1}(M)\leq \sigma(S^3)\left(\frac{\chi(\Sigma)}{2\sqrt{|c_1(L,\Sigma)}|}\right)^{\frac{4}{3}}.$$
\item If $\chi(\Sigma)>0$ and $c_1(L,\Sigma)=0$, then $\sigma^{S^1}(M)=\infty$.
\item If $\chi(\Sigma)\leq 0$, then $\sigma^{S^1}(M)=0$.
\end{enumerate}
\end{theorem}

In particular, $\sigma^{S^1}(M)$ is positive if and only if $\chi(\Sigma)$ is positive. This coincides with the characterization in~\cite{berardbergery},
as explained in Section~\ref{pos.mu.s1}.

\begin{proof}[Proof of Theorem \ref{mainthm}]
Let $[\tilde g]^{S^1}\in\mathrm{Conf}^{S^1}(M)$ be the class of 
$S^1$-invariant metrics conformal to $\tilde g$ on $M$. 
Without loss of generality, we assume that the length of the 
vector field $K$ generating the $S^1$-action $\ell:=|K|_{\tilde g}$ is 
constant (otherwise we take a different representant of the 
class $[\tilde g]^{S^1}$). Let $g$ be the projection of the 
metric~$\tilde g$ on $\Sigma$, so that $(M,\tilde g)\to (\Sigma,g)$ is a 
Riemannian submersion. Since $\ell$ is constant, 
the O'Neill formula \eqref{oneillscal} yields that 
$\scal_{\tilde g}=\scal_g-\frac{\ell^2}{4}|\Omega|_g^2$. Using the Gau\ss--Bonnet theorem, we compute the 
Yamabe functional as follows:
\begin{equation}\label{yamabe-M}
 \begin{split}
  J(\tilde g)
  &=\frac{2\pi\int_{\Sigma}(\scal_g-\frac{\ell^2}{4}|\Omega|_g^2)\ell \di v_g}{(2\pi)^{\frac{1}{3}}(\int_{\Sigma} \ell \di v_g)^{\frac{1}{3}}}\\
  &=(2\pi)^{\frac{2}{3}}\frac{\ell(\int_{\Sigma}\scal_g\di v_g)-\frac{\ell^3}{4}(\int_{\Sigma}|\Omega|^2_{g}\di v_g)}{\ell^{\frac{1}{3}}(\int_{\Sigma}\di v_g)^{\frac{1}{3}}}\\
  &=\left(\frac{\pi^2}
  {16\mathrm{vol}(\Sigma, g)}\right)^{\frac{1}{3}}
 \bigl(16\pi\chi(\Sigma)
  \ell^{\frac{2}{3}}-\|\Omega\|^2_2\ell^{\frac{8}{3}}\bigr).
 \end{split}
\end{equation}
If we have $\chi(\Sigma)\geq 0$ and  $\|\Omega\|_{2}>0$, 
then the maximal value of this expression
as a function in $\ell$ is attained for 
$\ell=\sqrt{4\pi\chi(\Sigma)}\|\Omega\|_{2}^{-1}$ and its maximal value equals 
$$3 \cdot 2^{\frac{4}{3}}\pi^2(\mathrm{vol}(\Sigma, g))^{-\frac{1}{3}
}\chi(\Sigma)^{\frac{4}{3}}\|\Omega\|_2^{-\frac{2}{3}}.$$ We now consider cases i) to iii) in the theorem.
\begin{enumerate}[label=\roman*), leftmargin=0.5cm]
\item Note that in this case  $c_1(L,\Sigma)\neq 0$ implies $\|\Omega\|_{2}>0$. 
By the Cauchy--Schwarz inequality, it further follows that
\begin{equation}\label{ineq-M}
   J(\tilde g)\leq 3 \cdot 2^{\frac{4}{3}}\pi^2\chi(\Sigma)^{\frac{4}{3}}\|\Omega\|_1^{-\frac{2}{3}}.
\end{equation}
On the other hand, we claim that $\|\Omega\|_{1}\geq 2\sqrt{2}\pi |c_1(L,\Sigma)|$, since  
\[\frac{1}{\sqrt{2}}\int_{\Sigma} |\Omega|_g \di v_g \geq \biggl|\int_{\Sigma} \Omega\biggr|=2\pi |c_1(L,\Sigma)|,\]
where the volume form $\di v_g$ has length $\sqrt 2$, by convention.
Using $\si(S^3)=3\cdot 2^{5/3}\cdot \pi^{4/3}$ it follows that 
$J(\tilde g)\leq \sigma(S^3)\chi(\Sigma)^{\frac{4}{3}}|4 c_1(L,\Sigma)|^{-\frac{2}{3}}$, for all $S^1$-invariant metrics $\tilde g$ on $M$ with $\ell=|K|_{\tilde g}$ constant. This yields 
$$ \mu(M,[\ti g]^{S^1})\leq \sigma(S^3)\chi(\Sigma)^{\frac{4}{3}}|4 c_1(L,\Sigma)|^{-\frac{2}{3}},$$
for all $S^1$-invariant conformal classes $[\tilde g]^{S^1}\in\mathrm{Conf}^{S^1}(M)$. 

Now, we show that $\sigma^{S^1}(M)$ is positive. 
The function $f:=\frac{2\pi}{vol(\Sigma,g)}\chi(\Sigma)-\frac{1}{2}\scal_g$ has zero average over $\Sigma$. By Lemma \ref{lem laplace}, there exists a solution $u$ of the equation $\Delta_g u=f$. Therefore the scalar curvature of 
$g_u:=e^{2u}g$ is given by
\begin{equation}
\scal_{g_u}= 2e^{-2u}(\Delta_gu+\frac{1}{2}\scal_g)=\frac{4\pi}{\vol(\Sigma,g)}\chi(\Sigma)e^{-2u}.
\end{equation}
Hence, the scalar curvature of $g_u$ is positive.
Using the identity \eqref{oneillscal} and choosing the length of the $S^1$-fibre constant and sufficiently small, we construct an $S^1$-invariant metric $\tilde g_u$ (which is not necessarily conformal to $\tilde g$) with positive scalar curvature. Therefore, the Yamabe constant $\mu(M,[\ti g_u]^{S^1})$ is positive, so $\sigma^{S^1}(M)>0$.

\item If $c_1(L,\Sigma)=0$, then there exists an $S^1$-equivariant finite covering $S^1\times \witi \Sigma$ of $M$ of degree $d$, where $\witi\Sigma$ is a smooth compact surface finitely covering $\Sigma$ (for more details, see \emph{e.g.} \cite[Lemma 3.7]{scott:83}). 
Since \mbox{$\chi(\Sigma)>0$}, we see that $\witi \Sigma$ is diffeomorphic to $S^2$. As in the previous case, we know that a metric of positive Gau{\ss} curvature exists on $\Sigma$. The product metric $\ti g_{\ell}$ of its lift to $\witi \Sigma$ with a rescaled standard metric on $S^1$ of length $2\pi\ell$ is invariant under the deck transformation group of $S^1\times \witi \Sigma\to M$. As this deck transformation group commutes with the $S^1$-action, $\ti g_{\ell}$ descends to an $S^1$-invariant metric $g_{\ell}$ on $M$. From \eqref{Yamfunctional}, we get $\mu(S^1\times \witi \Sigma, [\ti g_{\ell}]^{S^1})=\mu(S^1\times \witi \Sigma, [\ti g_{1}]^{S^1})\ell^{2/3}$. Obviously we have $\mu(S^1\times \witi \Sigma, [\ti g_{\ell}]^{S^1})\leq d^{2/3}\mu(M, [g_{\ell}]^{S^1})$.
Then $\mu(M, [g_{\ell}]^{S^1})$ converges to $\infty$ for $\ell \to \infty$, which implies the statement.
\item Assume
that the Euler-Poincar\'e characteristic of $\Sigma$ is nonpositive. 
By \eqref{yamabe-M}, we have
$$\mu(M,[\hat g]^{S^1})\leq J(\hat\ell^{-2} \hat g)\leq 2(2\pi)^{\frac{5}{3}}
\chi(\Sigma)\vol(\Sigma, \hat g_\Sigma)^{-\frac{1}{3}}\leq 0,$$ 
for any $S^1$-invariant Riemannian metric  $\hat g$ on $M$, where $\hat\ell:=|K|_{\hat g}$. This yields 
$\sigma^{S^1}(M)\leq 0$.
Moreover,  if we fix a Riemannian metric $\hat g_\Sigma$ on $\Sigma$, we 
define $(\hat g_j)$ to be a sequence of metrics on $M$ with 
constant functions $\hat \ell_j:=|K|_{\hat g_j}\leq 1$  converging to $0$ and 
$\pi^*\hat g_\Sigma=\hat g_j$. From \eqref{Yamfunctional} and using the H\"older inequality, we obtain 
 $$\mu(M,[\hat g_j]^{S^1})\geq - (2\pi\hat\ell_j)^{\frac{2}{3}}
(\|\Scal_{\hat g_{\Sigma}}\|_{\frac{3}{2}}+\frac{1}{4}\|\Omega\|^2_{3}).$$
Hence, when $j$ goes to $+\infty$, it follows that  $\sigma^{S^1}(M)\geq 0$.
We conclude that $\sigma^{S^1}(M)=0$.
\end{enumerate}
\end{proof}

\subsection{The case of $S^3$}\label{subsec S3}
We now consider the special case of $S^1$-actions on $S^3\subset\mC^2$.

For $m_1,m_2\in\mN$ assumed to be relatively prime as long as $m_1m_2\neq 0$, we define
\begin{equation}
\phi_{m_1,m_2}:S^1\to \mathrm{Diff}(S^3), \quad \phi_{m_1,m_2}(x)(z_1,z_2):=(x^{m_1} z_1, x^{m_2} z_2).
\end{equation}
With this notation, the Hopf action of $S^1$ on $S^3$ corresponds to $\phi_{1,1}$. These are the only possible smooth $S^1$-actions on $S^3$ up to diffeomorphisms (see {\it e.g.} \cite{orlik:72}). Note that such an action has fixed points if and only if $m_1m_2=0$. 

\begin{theorem}\label{S3}
For the Hopf action of $S^1$ on $S^3$ it holds: $$\sigma^{S^1}(S^3)=\sigma(S^3).$$
Moreover, the $S^1$-equivariant Yamabe invariant of any $S^1$-action $\phi_{m_1,m_2}$ on $S^3$ satisfies the following:
\begin{enumerate}[label=\roman*)]
\item If $m_1m_2=0$, then 
$\sigma^{S^1}(S^3)=\sigma(S^3)=6\cdot 2^{\frac{2}{3}}\cdot\pi^{\frac{4}{3}}$.
\item\label{item S3} If $m_1m_2\neq 0$, then 
\[\sigma(S^3)\leq \sigma^{S^1}(S^3)\leq \sigma(S^3)\biggl(\frac{m_1+m_2}{2\sqrt{m_1m_2}}\biggr)^{\frac{4}{3}}.\]
\end{enumerate} 
\end{theorem}

\begin{proof}
Let us first remark that, since the standard metric $g_{\mathrm{st}}$ on $S^3$ is $S^1$-invariant for any $S^1$-action $\phi_{m_1,m_2}$, it follows that $\mu(S^3,[g_{\mathrm{st}}]^{S^1})\geq \mu(S^3,[g_{\mathrm{st}}])=\sigma(S^3)$. Hence, we obtain the inequality: 
$\sigma^{S^1}(S^3)\geq \sigma(S^3)$.

\begin{enumerate}[label=\it\roman*), leftmargin=0.4cm]
\item If $m_1m_2=0$, then the $S^1$-action has fixed points and by Proposition \ref{Prop-HV} we also obtain the reverse inequality: $\sigma^{S^1}(S^3)\leq \sigma(S^3)$.
\item If $m_1m_2\neq 0$, then the quotient orbifold is the so-called $1$-dimensional weighted projective space denoted by $\mC P^1(m_1,m_2)$. In order to use the upper bound provided by Theorem \ref{mainthm}, we need to compute $\chi(\mC P^1(m_1,m_2))$ and $c_1(L,\mC P^1(m_1,m_2))$. Using the Seifert invariants of $S^1$-bundles (see {\it e.g.} \cite{orlik:72}, \cite{scott:83}), one obtains: $\chi(\mC P^1(m_1,m_2))=\frac{1}{m_1}+\frac{1}{m_2}$ and $|c_1(L,\mC P^1(m_1,m_2))|=\frac{1}{m_1m_2}$.  Alternatively, we give in the Appendix an explicit geometric computation of this topological invariants. Substituting these values in Theorem \ref{mainthm}, \ref{item i}, we obtain the desired inequality.
\end{enumerate}
The first statement of the theorem follows from $\ref{item S3}$ for $m_1=m_2=1$.
\end{proof}

\section{Convergence result}

\begin{definition}
Let $G$ be a Lie group, and let $(H_i)_{i\in \mN}$ be a sequence of subgroups.
We say that $h\in G$ is an accumulation point for  $(H_i)_{i\in \mN}$ if there 
is a sequence $(h_i) _{i\in \mN}$ with $h_i\in H_i$ and $h_i\to h$. The set of accumulation points is a closed subgroup of~$G$.
We say that $(H_i)_{i\in \mN}$ is accumulating, if every element of $G$ is an accumulation point.
\end{definition}

\begin{proposition}
Assume that a compact Lie group $G$ acts on a closed manifold~$M$.
Let $(H_i)_{i\in \mN}$ be an accumulating sequence of subgroups of $G$.
Then for any $G$-equivariant conformal class $[g]$ we get
  $$\lim_{i\to \infty}\mu(M,[g]^{H_i})= \mu(M,[g]^{G}).$$
\end{proposition}

\begin{proof}
We distinguish the following two cases: 
\begin{itemize}[leftmargin=0.2cm]
\item If the (non equivariant) Yamabe constant satisfies $\mu(M,[g])\leq 0$, then there is, up to a multiplicative constant, 
a unique metric $u_{\infty}^{\frac{4}{n-2}}g$ of constant scalar curvature 
and $u_\infty$ is $G$-invariant. 
This implies $\mu(M,[g])=\mu(M,[g]^{G})=\mu(M, [g]^{H_i})$. 

\item  Now we assume that the Yamabe constant satisfies $\mu(M,[g])>0$. Set $\mu_i:=\mu(M,[g]^{H_i})$, $\mu:=\mu(M,[g]^{G})$. Obviously $\mu_i\leq \mu$. After passing 
to a subsequence we can assume that $\mu_i$ converges to a number 
$\bar\mu\leq \mu$ and it remains to show that $\bar\mu<\mu$ leads to a contradiction.  For an orbit $O$ we will use the convention that $\# O$ takes values in $\mN\cup\{\infty\}$, i.e. we do not distinguish between
infinite cardinalities. 
We claim that $\lim_{i\to\infty} \#(H_i\cdot p)=\#(G\cdot p)$, for any $p\in M$.
The inequality $\#(H_i\cdot p)\leq\#(G\cdot p)$ is obvious as $H_i\subset G$.

To prove the claim in the case $\#(G\cdot p)<\infty$, we choose pairwise disjoint neighborhoods of all the $G$-orbit points of $p$ and for $i$ sufficiently large, we find in each such neighborhood an element of the $H_i$-orbit of $p$, showing that $ \#(H_i\cdot p)\geq\#(G\cdot p)$. If $\#(G\cdot p)=\infty$, then we apply the previous argument to a finite subset of the $G$-orbit of $p$ and then let its cardinality converge to $\infty$. This shows that $\lim_{i\to\infty} \#(H_i\cdot p)=\infty$.

Without loss of generality, we assume that
$\mu_i\leq \tilde\mu:=(\mu+\bar\mu)/2<\mu$.
Let~$k$ be the cardinality of the smallest $G$-orbit, again sloppily 
written as $\infty$ in the case that $k$ is infinite. Then by Proposition \ref{Prop-HV}, we have $\mu\leq \sigma(S^n)k^{2/n}.$
This implies  \mbox{$\mu_i\leq \tilde\mu  <\sigma(S^n) k^{2/n}$}. Hence, by the above claim, we obtain the following inequality $\mu_i\leq \tilde\mu  <\sigma(S^n) (\min_{p\in M} \#(H_i\cdot p))^{2/n}$, for $i\geq i_0$, where $i_0$ is sufficiently large.
By Hebey and Vaugon~\cite{hebey.vaugon:93}, it follows that, for $i\geq i_0$, 
there exists  a sequence $(u_i^{\frac{4}{n-2}}g)_{i\in\mN}$ of $H_i$-invariant 
metrics, which minimizes the functional $J$ among all $H_i$-invariant metrics in $[g]$.
Furthermore $u_i$ is a positive smooth $H_i$-invariant  
solution of the Yamabe equation, and we may assume  $\|u_i\|_{\frac{2n}{n-2}}=1$. 
The sequence $(u_i)_{i\in\mN}$ is uniformly bounded in $H^1(M)$. 
Hence there exists a nonnegative function $u_\infty\in H^1(M)$, such that $(u_i)_{i\in\mN}$ converges strongly in $L^q(M)$, for $1\leq q<\frac{2n}{n-2}$, and weakly in $H^1(M)$ to $u_\infty$. 
We now claim, that $u_i$ is bounded in the $L^\infty$-norm. Suppose that it is not bounded. Then we find a sequence of $x_i\in M$ such that $u_i(x_i)\to \infty$, and after taking a subsequence $x_i$ converges to a point $\bar x$. For any point $g\bar x$ in its orbit, there is a sequence of $h_i\in H_i$ with 
$h_ix_i\to g\bar x$, $u_i(h_ix_i)\to \infty$. If the orbit $G\cdot \bar x$ contains at least $\ti k$ points,  
then we can do classical blowup-analysis in $\ti k$ points, which would  yield $\bar\mu\geq  \sigma(S^n) \ti k^{2/n}$ (see for example \cite[Chapter 6.5.]{aubin:98}). This implies $\bar\mu\geq  \sigma(S^n) k^{2/n}$ which contradicts
$\bar  \mu<\sigma(S^n) k^{2/n}$. We obtain the claim, i.e.\ the boundedness
of $u_i$ in $L^\infty$. By a standard bootstrap argument this yields
the boundedness of $u_i$ in $C^{2,\al}$ for $0<\alpha<1$, and thus 
$u_i$ converges to $u_\infty$ in $C^2$. It follows that $u_\infty$ is a smooth, positive $G$-invariant solution of the Yamabe equation, 
with $\|u_\infty\|_{\frac{2n}{n-2}}=1$ and $J(u_\infty^{\frac{4}{n-2}}g)=\bar\mu$. 
Thus $\mu\leq\bar{\mu}$.
\end{itemize}
\end{proof}

\begin{corollary}\label{approx.cor}
Assume that a compact Lie group $G$ acts on a closed manifold $M$.
Let $(H_i)_{i\in \mN}$ be an accumulating sequence of subgroups of $G$.
Then 
  $$\liminf_{i\to \infty}\si^{H_i}(M)\geq  \si^{G}(M).$$
\end{corollary}
\qed
\appendix
\section{}

\subsection{Computation of $\boldsymbol{c_1(L, \mC P^1(m_1,m_2))}$ }

We consider  the action of $S^1$ on $S^3\subset \mC^2$ given by $$\phi_{m_1,m_2}:e^{i\theta}\mapsto \bigl((z_1,z_2)\mapsto (e^{im_1\theta}z_1,
e^{im_2\theta}z_2)\bigr),$$ 
where $m_1$ and $m_2$ are two positive relatively prime  integers. Let $\pi: S^3\to S^3/S^1$ denote the projection, where the quotient $S^3/S^1=:\mC P^1(m_1,m_2)$ is the one dimensional weighted projective space. We consider the round metric of $S^3$ induced by the standard metric on $\mR^4\simeq \mC^2$: $\<(z_1,z_2),(w_1,w_2)\>=\textrm{Re}(z_1\bar w_1+z_2\bar w_2)$ . The vector field induced by the $S^1$-action is given by:
\[K_{p}=i(m_1z_1,m_2z_2)\in T_pS^3=p^\perp, \text{ where } p=(z_1,z_2)\in S^3.\] 
The vector field $K$ vanishes nowhere, since $|K_p|^2=m_1^2|z_1|^2+m_2^2|z_2|^2>0$, for all $p\in S^3$. For $p\in S^3\setminus (\{0\}\times S^1\cup S^1\times \{0\})$, the orthogonal complement of $K_p$ in $T_pS^3$ (w.r.t. the round metric) is spanned by the horizontal vector fields
$$ \witi X_1(p):=i(m_2|z_2|^2z_1, -m_1|z_1|^2z_2), \quad  \witi X_2(p):=(|z_2|^2z_1, -|z_1|^2z_2), $$
which are also $S^1$-invariant. Hence they project to the vector fields $X_1$, resp. $X_2$ on $\mC P^1(m_1,m_2)$. 

We define the connection $1$-form $\omega$ on $S^3$ whose kernel is given by the orthogonal complement of $K$ and normed such that $\omega(K)=1$, $\omega:=\frac{\<K,\cdot\>}{|K|^2}$. The $2$-form $\Omega:=\di \om$ is $S^1$-invariant and thus projects onto a $2$-form on $\mC P^1(m_1,m_2)$, which we denote by the same symbol. It follows that
\[\Omega_{\pi(p)}(X_1,X_2)=-\omega_p([\witi X_1, \witi X_2])=\frac{2m_1m_2|z_1|^2|z_2|^2}{m_1^2|z_1|^2+m_2^2|z_2|^2},\]
since we have $\di\witi X_1(\witi X_2)-\di\witi X_2(\witi X_1)=-2i|z_1|^2|z_2|^2(m_2z_1,m_1z_2)$.

We now introduce the following complex coordinates on $\mC P^1(m_1,m_2)\setminus\{[0:1]\}$. 
\begin{eqnarray*}
\varphi: \mC P^1(m_1,m_2)\setminus \{[0:1]\} & \longrightarrow & \mC \\
    {[ z_1 : z_2 ]} & \longmapsto & z:=\frac{z_2^{m_1}}{z_1^{m_2}}.
\end{eqnarray*}
It follows that for any 
$p\in S^3\setminus(\{0\}\times S^1)$, the tangent linear map of the projection is given by 
\[\pi_*(p)=\biggl(-m_2\frac{z_2^{m_1}}{z_1^{m_2+1}}, m_1\frac{z_2^{m_1-1}}{z_1^{m_2}}\biggr)\]
and the vector fields $X_1$ and $X_2$ are 
\[X_1(z)=-(m_2^2|z_2|^2+m_1^2|z_1|^2)iz, \quad X_2(z)=-(m_2|z_2|^2+m_1|z_1|^2)z.\]
These together imply the following:
\begin{eqnarray*}
\Omega_z & = &\frac{-m_1m_2|z_1|^2|z_2|^2}{(m_2^2|z_2|^2+m_1^2|z_1|^2)^2(m_2|z_2|^2+m_1|z_1|^2)|z|^2}i\di z\wedge \di\bar z,\\
c_1(L, \mC P^1(m_1,m_2)) & = & \frac{1}{2\pi}\int_{\mC P^1(m_1,m_2)}\hspace{-0.7cm}\Omega= -\int_0^\infty \hspace{-0.4cm}\frac{2m_1m_2r(1-r)}{(m_2^2+(m_1^2-m_2^2)r)^2(m_2+(m_1-m_2)r)}\frac{\di\rho}{\rho}\\
& = & \int_0^1\frac{m_1m_2}{(m_2^2+(m_1^2-m_2^2)r)^2}\di r=\frac{1}{m_1m_2},
\end{eqnarray*}
where $r=|z_1|^2$, $\rho=|z|$ and $\rho=\frac{(1-r)^{\frac{m_1}{2}}}{r^{\frac{m_2}{2}}}$. 

\subsection{Computation of $\boldsymbol{\chi(\mC P^1(m_1,m_2))}$ }

The quotient metric $g$ induced on $\mC P^1(m_1,m_2)$ by the standard metric of $S^3$ is uniquely determined by setting that the following two vector fields  of the tangent space of $\mC P^1(m_1,m_2)$ at $z\in\mC\setminus\{0\}$ build an orthonormal base:
\[e_1(z):=\frac{X_1(z)}{|\witi X_1(z)|}=\lambda_1(z)iz, \quad e_2(z):=\frac{X_2(z)}{|\witi X_2(z)|}=\lambda_2(z)z,\]
where $\lambda_j(z):=\witi\lambda_j
\circ\gamma^{-1}(|z|^2)$,  $\witi\lambda_1(t):=-\frac{\sqrt{(m_1^2-m_2^2)t+m_2^2}}{\sqrt{t(1-t)}}$, $\witi\lambda_2(t):=-\frac{(m_1-m_2)t+m_2}{\sqrt{t(1-t)}}$ and $\gamma$ is the diffeomorphism $\gamma(r):=\frac{(1-r)^{m_1}}{r^{m_2}}$, for $r\in(0,1)$ and $|z|^2=\frac{(1-|z_1|^2)^{m_1}}{|z_1|^{2m_2}}=\gamma(|z_1|^2)$.
We consider $\Theta$ to be the Levi-Civita connection $1$-form. We have
$$\Theta(v):=g(\nabla_v e_2, e_1)=g([e_1,e_2],v).$$
We first compute the Lie bracket:
\begin{equation*}
[e_1,e_2]  =   \lambda_1  \lambda_2[iz,z]+ \lambda_1 \di \lambda_2(iz)z- \lambda_2 \di \lambda_1(z)iz=-\frac{ \di \lambda_1(e_2)}{ \lambda_1}e_1,
\end{equation*}
since $[iz,z]=0$ and $\di  \lambda_j=2(\witi\lambda_j\circ
\gamma^{-1})'(|\cdot|^2)z$, for $j=1,2$, which implies $\di \lambda_2(iz)=0$.
Secondly, we compute the Gaussian curvature of $\mC P^1(m_1,m_2)$:
\begin{multline*}
\kappa=\di\Theta(e_1,e_2)=-\di(g([e_1,e_2],e_1))(e_2)-\Theta([e_1,e_2])=\di\biggl(
\frac{ \di \lambda_1(e_2)}{ \lambda_1}\biggr)(e_2)-\biggl(\frac{ \di \lambda_1(e_2)}{ \lambda_1}\biggr)^2
\end{multline*}
Hence $\di\Theta=\kappa e_1^*\wedge e_2^*=-\frac{\kappa}{|\cdot|^2 \lambda_1 \lambda_2 }\di x\wedge \di y$. By the orbifold Gau\ss--Bonnet theorem, it follows that
\[\chi(\mC P^1(m_1,m_2))=\frac{1}{2\pi}\int_{\mC P^1(m_1,m_2)} \di\Theta=\frac{-1}{2}\int_0^\infty \frac{\kappa}{|\cdot|^2 \lambda_1 \lambda_2}  \di |z|^2=
\frac{1}{2}\int_0^1\frac{\kappa(r)\gamma'(r)}{\witi\lambda_1(r)\witi\lambda_2(r)\gamma(r)}\di r,\]
since the functions $\lambda_j$ are radial and thus $\kappa$ is also radial.
Substituting $\kappa$ in the last integral, we get
\begin{equation*}
\begin{split}
\chi(\mC P^1(m_1,m_2))= &  {2}\int_0^1\biggl(\biggl(\frac{\witi\lambda_2\witi\lambda'_1\gamma}{\witi\lambda_1\gamma'}\biggr)'\frac{\gamma\witi\lambda_2}{\gamma'}-\biggl(\frac{\witi\lambda_2\witi\lambda'_1\gamma}{\witi\lambda_1\gamma'}\biggr)^2\biggr)\frac{\gamma'}{\witi\lambda_1\witi\lambda_2\gamma}\di r\\
=  & {2}\int_0^1\biggl(\frac{\witi\lambda_2\witi\lambda'_1\gamma}{\witi\lambda_1\gamma'}\biggr)'
\frac{1}{\witi\lambda_1}-
\frac{\witi\lambda_2 (\witi\lambda_1')^2\gamma}{\witi\lambda_1^3\gamma'}\di r
=  {2}\biggl[\frac{\witi\lambda_2\witi\lambda'_1\gamma}{\witi\lambda^2_1\gamma'}\biggr]_0^1=\frac{1}{m_1}+\frac{1}{m_2}.
\end{split}
\end{equation*}

\bibliographystyle{amsplain}
\bibliography{equiyamabe}

\end{document}